\documentclass[11pt,reqno]{amsart}
\usepackage[utf8]{inputenc}
\usepackage{amsmath}
\usepackage{amssymb}
\usepackage[usenames,dvipsnames]{pstricks}
\usepackage{hyperref}

\numberwithin{equation}{section}
\newtheorem{thm}{Theorem}[section]
\newtheorem{cor}[thm]{Corollary}
\newtheorem{lem}[thm]{Lemma}
\newtheorem{prop}[thm]{Proposition}

{\theoremstyle{remark}
\newtheorem{example}[thm]{Example}}
\newcommand{\abs}[1]{\left\vert#1\right\vert}
\newcommand{\Reg}{\mathop{\mathrm{Reg}}}
{\theoremstyle{definition}
\newtheorem{algorithm}[thm]{Algorithm}}

\newcommand{\malcev}{\mathbin{\hbox{$\bigcirc$\rlap{\kern-8.25pt\raise0,50pt\hbox{${\tt
  m}$}}}}}
\newcommand{\smalcev}{\mathbin{\hbox{$\bigcirc$\rlap{\kern-7pt\raise0,30pt\hbox{${\tt
  m}$}}}}}
\begin{document}
\title{The pro-nilpotent group topology on a free group}
\author{J. Almeida, M.H. Shahzamanian and B. Steinberg}
\address{J. Almeida and M.H. Shahzamanian\\ Centro de Matem\'atica e Departamento de Matem\'atica, Faculdade de Ci\^{e}ncias,
Universidade do Porto, Rua do Campo Alegre, 687, 4169-007 Porto,
Portugal}
\email{jalmeida@fc.up.pt; m.h.shahzamanian@fc.up.pt}
\address{B. Steinberg\\ Department of Mathematics, City College of New York, New York City, NY 10031.}
\email{bsteinberg@ccny.cuny.edu\footnote{Corresponding author}}

\begin{abstract}
In this paper, we study the pro-nilpotent group topology on a free group. First we describe the closure of the
product of finitely many finitely generated subgroups of a free group in the pro-nilpotent group topology and then present an algorithm to compute it. We deduce that the nil-closure of a rational subset of a free group is an effectively constructible rational subset and hence has decidable membership.  We also prove that the $\mathsf{G_{nil}}$-kernel of a finite monoid is computable and hence pseudovarieties of the form $\mathsf{V}
\smalcev \mathsf{G_{nil}}$ have decidable membership problem, for every decidable pseudovariety of monoids $\mathsf{V}$.  Finally, we prove that the semidirect product $\mathsf{J} \ast \mathsf{G_{nil}}$ has a decidable membership problem.
\end{abstract}
\maketitle
%\tableofcontents
%%%%%%%%%%%%%%%%%%%%%%%%%%%%%%%%%%%%%%%%%%%%%%%%%%%%%%%%%%%%%%%%%%%%%%%%%%%%%%%%%%%%%%%%%%%%%%%%%%%%%%%%%%%%%%%%%%%%%%%%%%%%%%%%%%%%%%%%%%%%%%%%%%%%%%%%%%%%%%%%%%%%%%%%%%%%%%%%%%%%%%%%%%%%%%%%%%%%%%%%%%%%%%%%%%%%%%%%%%%%%%%%%%%%%%%%%%%%%%%%%%%%%%%%%%%%%%%%%%%%%%%%%%%%%%%%%%%%%%%%%%%%%%%%
%%%%%%%%%%%%%%%%%%%%%%%%%%%%%%%%%%%%%%%%%%%%%%%%%%%%%%%%%%%%%%%%%%%%%%%%%%%%%%%%%%%%%%%%%%%%%%%%%%%%%%%%%%%%%%%%%%%%%%%%%%%%%%%%%%%%%%%%%%%%%%%%

\section{Introduction}\label{pre}

Hall showed that every finitely generated subgroup of a free group is closed in the profinite topology~\cite{Hal}. Pin and Reutenauer conjectured that if $H_1, H_2, \ldots , H_n$ are finitely generated subgroups of a free group, then the product $H_1H_2\cdots H_n$ is closed in the profinite topology~\cite{Pi-Ru}. Assuming this conjecture they presented a simple algorithm to compute the closure of a given rational subset of a free group. This conjecture was proved by Ribes and Zalesski\u\i~\cite{RiZa0}, who also later proved that if the subgroups $H_1, \ldots, H_n$ are $p$-closed for some prime $p$,
then $H_1 \cdots H_n$ is $p$-closed, too~\cite{RiZa}.

Margolis, Sapir and Weil provided an algorithm to compute the nil-closure of a finitely generated subgroup of a free group~\cite{Mar01}.
In this paper, we provide an example of two finitely generated nil-closed subgroups $H,K$ of a free group whose product $HK$ is not nil-closed.  However, we prove that the nil-closure of the product of finitely many finitely generated subgroups is the intersection over all primes $p$ of its $p$-closures and present a procedure to compute algorithmically a finite automaton that accepts precisely the reduced words in the nil-closure of the product.
Hence, there is a uniform algorithm to compute membership in the nil-closure of a product of finitely many finitely generated subgroups of a free group. We also prove that the nil-closure of a rational subset of a free group is again a rational subset and provide an algorithm to compute an automaton that accepts the reduced words in the nil-closure.
%Based on this result, we prove that the nil-closure of a rational subset of a free monoid is computable.
This yields that the Mal'cev product $\mathsf{V} \malcev \mathsf{G_{nil}}$ is decidable for every decidable pseudovariety of monoids $\mathsf{V}$, where $\mathsf{G_{nil}}$ is the pseudovariety of all finite nilpotent groups.

Auinger and the third author introduced the concept of arboreous pseudovarieties of groups~\cite{Auinger&Steinberg:2001a} and proved that a pseudovariety of groups $\mathsf{H}$ is arboreous if and only if $\mathsf{J}\malcev \mathsf{H} = \mathsf{J}\ast\mathsf{H}$, where $\mathsf{J}$ is the pseudovariety of all finite $\mathcal{J}$-trivial monoids. The pseudovariety $\mathsf{G_{nil}}$ is not arboreous and, therefore, $\mathsf{J}\malcev \mathsf{G_{nil}}\neq\mathsf{J}\ast \mathsf{G_{nil}}$. We prove that the pseudovariety $\mathsf{J}\ast \mathsf{G_{nil}}$ has decidable membership as an application of our results on computing nil-closures of rational subsets.

\section{Decidability of the nil-closure of rational subsets}
Let $G$ be a group and $\mathsf{H}$ a pseudovariety of groups, that is, a class of finite groups closed under finite direct products, subgroups and homomorphic images. Then the pro-$\mathsf{H}$ topology on $G$ is the group topology  defined by taking as a fundamental system of neighborhoods of the identity all normal subgroups $N$ of $G$ such that $G/N \in \mathsf{H}$.  This is the weakest topology on $G$ so that every homomorphism of $G$ to a group in $\mathsf{H}$ (endowed with the discrete topology)
is continuous. We say that $G$ is \emph{residually} $\mathsf{H}$ if, for every $g \in G\setminus \{1\}$, there is a homomorphism $\varphi\colon G \rightarrow H \in\mathsf{H}$ with $\varphi(g)\neq 1$, or, equivalently, $\{1\}$ is an $\mathsf H$-closed subgroup. In this case, the pro-$\mathsf{H}$ topology is Hausdorff and, in fact, it is metric when $G$ is finitely generated. More precisely, when $G$ is finitely generated, the topology is given by the following ultrametric \'ecart. For $g \in G$, define
$$r(g) = \min(\{[G : N]\mid G/N \in \mathsf{H}, g\not\in  N\} \cup \{\infty\}).$$
Then the $\mathsf{H}$-pseudonorm is given by
$$\abs{g}_{\mathsf{H}} = 2^{-r(g)} \mbox{ (where }2^{-\infty} = 0).$$
One can verify that
$$\abs{g_1g_2}_{\mathsf{H}} \leq \max\{\abs{g_1}_{\mathsf{H}}, \abs{g_2}_{\mathsf{H}}\}.$$
For $g_1, g_2 \in G$, we define
$$d_{\mathsf{H}}(g_1, g_2) = \abs{g_1g_2^{-1}}_{\mathsf{H}}.$$
It is easy to see that this is an ultrametric \'ecart defining the
pro-$\mathsf{H}$ topology, which is a metric if and only if $G$ is
residually $\mathsf{H}$. The pseudovariety $\mathsf{H}$ is said to be
\emph{extension-closed} if, whenever $1 \rightarrow N \rightarrow G
\rightarrow H \rightarrow 1$ is an exact sequence of groups with $N,H
\in \mathsf{H}$, we also have $G \in \mathsf{H}$.

We use $\mathrm{Cl}_{\mathsf{H}}(X)$ to denote the closure of $X\subseteq G$ in the pro-$\mathsf{H}$ topology. If $p$ is a prime, then $\mathsf{G_p}$ denotes the pseudovariety of all finite $p$-groups.  We denote by $\mathsf{G_{nil}}$ the pseudovariety of all finite nilpotent groups.
We talk of $p$-closure, $p$-denseness, etc., instead of $\mathsf{G_p}$-closure, $\mathsf{G_p}$-denseness, etc.
We also talk of nil-closure, nil-denseness, etc., instead of $\mathsf{G_{nil}}$-closure, $\mathsf{G_{nil}}$-denseness, etc.

An important property that we shall exploit is that if $\mathsf H\subseteq \mathsf K$ are pseudovarieties of groups, then $\mathrm{Cl}_{\mathsf H}(X)=\mathrm{Cl}_{\mathsf H}(\mathrm{Cl}_{\mathsf K}(X))$ for $X\subseteq G$.

\subsection{The nil-closure of a product of finitely generated subgroups}
In this subsection we describe the nil-closure of a finite product of subgroups of a free group in terms of its $p$-closures.  In the next subsection, we provide an algorithm to compute the nil-closure of a finite product of finitely generated subgroups, and more generally, of a rational subset.

Let $H$ be a finitely generated subgroup of a free group $F$. Margolis, Sapir and Weil proved that the nil-closure of $H$ is the intersection over all primes $p$ of the $p$-closures of $H$~\cite[Corollary 4.1]{Mar01}.  Our goal is to first prove an analogous result for products of subgroups of free groups.

The following straightforward lemma describes the $\mathsf{H}$-closure of an arbitrary subset of a group with respect to the pro-$\mathsf{H}$ topology for a pseudovariety $\mathsf{H}$ of groups.

\begin{lem} \label{CloX}
Let $F$ be a group with its pro-$\mathsf{H}$ topology, and let $X$ be a subset of $F$. Then
the $\mathsf{H}$-closure of $X$ is the intersection over all $H\in\mathsf{H}$ and (onto) homomorphisms $\varphi:F\rightarrow H$ of the subsets $\varphi^{-1}\varphi(X)$.
\end{lem}

Let $G$ be a group in the pro-$\mathsf{H}$ topology, $X \subseteq G$, and $g \in G$.
Note that $\mathrm{Cl}_{\mathsf{H}}(Xg) = \mathrm{Cl}_{\mathsf{H}}(X)g$ and $\mathrm{Cl}_{\mathsf{H}}(gX) = g \mathrm{Cl}_{\mathsf{H}}(X)$ as left and right translations by elements of $G$ are homeomorphisms.

\begin{prop} \label{04}
Let $F$ be a group and $H_1, \ldots, H_n$ be subgroups of $F$. Then the nil-closure of $H_1\cdots H_n$ is the intersection over all primes $p$ of the $p$-closures of $H_1\cdots H_n$.
\end{prop}

\begin{proof}
Suppose that $x \in \mathrm{Cl}_{nil}(H_1\cdots H_n)$. Then, for every homomorphism $\varphi \colon F \rightarrow G$ onto a finite nilpotent group $G$, one has $\varphi(x) \in \varphi(H_1 \cdots H_n)$. In particular, for every homomorphism $\varphi \colon F \rightarrow P$ onto a finite $p$-group $P$, we have that $\varphi(x) \in \varphi(H_1 \cdots H_n)$. Thus, $x \in \mathrm{Cl}_{p}(H_1\cdots H_n)$ for every prime $p$ and hence $x$ is in the intersection over all primes $p$ of the $p$-closures of $H_1\cdots H_n$.

Now suppose that an element $x$ is in the intersection of the $p$-closures of $H_1\cdots H_n$ over all primes $p$ and $\varphi \colon F \rightarrow G$ is a homomorphism onto a finite nilpotent group $G$.

The group $G$, being a finite nilpotent group, is the direct product of its Sylow subgroups (each of which is normal). Suppose that the Sylow subgroups of $G$ are $P_1, \ldots, P_m$ where
$P_j$ is a $p_j$-Sylow subgroup for a prime $p_j$ and $1 \leq j \leq m$;  then  $G=P_1\cdots P_m$.
The group $\varphi(H_i)$ is a subgroup of $G$, so it is nilpotent and, thus, it is also a direct product of its Sylow subgroups  for every $1 \leq i \leq n$. Therefore, $\varphi(H_i)=P_{i,1}\cdots P_{i,m}$ such that $P_{i,j}=1$ or $P_{i,j}$ is a $p_i$-Sylow subgroup of $\varphi(H_i)$, for every $1 \leq i \leq n$ and $1 \leq j \leq m$. As $G$ is nilpotent, it has a unique $p$-Sylow subgroup, for each prime divisor $p$ of its order, and so $P_{i,j} \subseteq P_j$, for $i=1,\ldots, n$ and $j=1,\ldots, m$.  As the Sylow subgroups of $G$ are normal and have pairwise trivial intersection, we have that $P_i$ and $P_j$ commute elementwise (for $i\neq j$).   Thus it follows that
\begin{equation}\label{eq:directproduct}
\begin{split}
\varphi(H_1 \cdots H_n) &=
 P_{1,1} \cdots P_{1,m}
  P_{2,1} \cdots P_{2,m}\cdots
   P_{n,1}\cdots P_{n,m}\\ &=P_{1,1} \cdots P_{n,1} P_{1,2} \cdots P_{n,2}\cdots P_{1,m}\cdots P_{n,m}.
   \end{split}
\end{equation}

Since $\varphi(x)\in G$, there exist unique $a_1 \in P_1,a_2 \in P_2, \ldots, a_m \in P_m$ such that $\varphi(x)=a_1a_2\cdots a_m$.
For every prime $p_j$, $1 \leq j \leq m$, consider the canonical projection $\pi_j\colon G \rightarrow P_j$. Since $\pi_j \varphi$ is a homomorphism and $x \in \mathrm{Cl}_{p_j}(H_1\cdots H_n)$, we conclude that $a_j=\pi_j\varphi(x) \in P_{1,j} \cdots P_{n,j}$ for $j=1,\ldots, m$. Hence, $\varphi(x) \in \varphi(H_1\cdots H_n)$  by~\eqref{eq:directproduct} and, therefore, the intersection over all primes $p$ of the $p$-closures of $H_1\cdots H_n$ is in the nil-closure of $H_1\cdots H_n$.
The result follows.
\end{proof}

This yields immediately the following corollary.

\begin{cor} \label{main-theorem}
Let $F$ be a group and $H_1, \ldots, H_n$ be subgroups of $F$.
\begin{enumerate}
\item If the subset $H_1 \cdots H_n$ is $p$-dense in $F$, for every prime $p$, then the subset $H_1 \cdots H_n$ is nil-dense in $F$.
\item If, for every prime $p$, there exists an integer $1 \leq i \leq n$ such that $H_i$ is $p$-dense, then the subset $H_1 \cdots H_n$ is nil-dense in $F$.
\end{enumerate}
\end{cor}

\begin{example}\label{ex:products.not.closed}
Consider the subgroups $H = \langle a^2, b \rangle$  and $K =\langle  a,b^3\rangle$ of the free group $F$
on the set $\{a, b\}$. By~\cite[Corollary 3.3]{Mar01}, the subgroup $H$ is $p$-dense for every prime except prime $2$ and the subgroup $K$ is $p$-dense for every prime except prime $3$.  Therefore, $HK$ is nil-dense by Corollary~\ref{main-theorem}.  Since $HK$ is a proper subset of $F$ (in fact, every pair of infinite index, finitely generated subgroups of a free group has infinitely many double cosets, cf.~\cite{isproper}), it follows that $HK$ is not nil-closed.

On the other hand, it is easily checked using the algorithm in~\cite[Section 3.2]{Mar01}, that the subgroup $H$ is $2$-closed and the subgroup $K$ is $3$-closed. Alternatively, $H$ is a free factor of the kernel of the mapping $F\to \mathbb Z/2\mathbb Z$ mapping $a$ to $1+2\mathbb Z$ and $b$ to $0+2\mathbb Z$, $K$ is a free factor of the kernel of the mapping $F\to \mathbb Z/3\mathbb Z$ mapping $a$ to $0+3\mathbb Z$ and $b$ to $1+3\mathbb Z$ and free factors of open subgroups are closed in the pro-$p$ topology, for any prime $p$, by the results of Ribes and Zalessk{\u\i}~\cite{RiZa,Mar01}.

It follows that $H$ and $K$ are both nil-closed. So a product of nil-closed subgroups of a free group is not necessarily nil-closed and hence it is not in general true that $\mathrm{Cl}_{nil}(H_1\cdots H_n)=\mathrm{Cl}_{nil}(H_1)\mathrm{Cl}_{nil}(H_2)\cdots \mathrm{Cl}_{nil}(H_n)$ for $n>1$.
\end{example}

\subsection{Computing the nil-closure of a rational subset}
If $M$ is a monoid (with the relevant examples for us being free groups and free monoids), then a subset of $M$ is said to be \emph{rational} if it belongs to the smallest collection $\mathcal C$ of subsets of $M$ such that:
\begin{itemize}
\item finite subsets of $M$ belong to $\mathcal C$;
\item if $X,Y\in \mathcal C$, then $X\cup Y$ belongs to $\mathcal C$;
\item if $X,Y\in \mathcal C$, then $XY=\{xy\mid x\in X,y\in Y\}$ belongs to $\mathcal C$;
\item if $X\in \mathcal C$, then the submonoid of $M$ generated by $X$ belongs to $\mathcal C$.
\end{itemize}
If $M$ is finitely generated by a set $A$ and $\pi\colon A^*\to M$ is the canonical projection (where $A^*$ denotes the free monoid on $A$), then a subset $X$ of $M$ is rational if and only if there is a regular language $L$ over $A$ such that $X=\pi(L)$.  Recall that a \emph{regular language} over an alphabet $A$ is a subset accepted by a finite $A$-automaton $\mathcal A$.  Here, we take a finite $A$-automaton to be a finite directed graph with edges labeled by elements of $A$ together with two distinguished subsets of vertices $I$ and $T$.  The language $L(\mathcal A)$ accepted by $\mathcal A$ consists of all words in $A^*$ labeling a path from a vertex in $I$ to a vertex in $T$. Regular languages over $A$ are exactly the rational subsets of $A^*$ by Kleene's theorem.  See Eilenberg's book~\cite{EilenbergA} for details.

If $G$ is a group, then a theorem of Anissimov and Seifert~\cite{Anisimov} says that a subgroup $H$ of $G$ is rational if and only if it is finitely generated.  Any finite product $H_1\cdots H_n$ of finitely generated subgroups $H_1,\ldots, H_n$ of $G$ is rational, as is any translate $gH_1\cdots H_n$.  Consequently, any finite union of translates of products of finitely many finitely generated subgroups of $G$ is rational.

Let $A$ be a finite set and $F(A)$ the free group on $A$. Let $\widetilde A=A\cup A^{-1}$ where $A^{-1}$ is a set of formal inverses of the elements of $A$ and denote by $\pi\colon \widetilde A^*\to F(A)$ the natural projection.  Then Benois proved~\cite{Beno} that a subset $L$ of $F(A)$ is rational if and only if the set of reduced words in $\widetilde A^*$ representing elements of $L$ (under $\pi$) is a regular language and consequently the rational subsets of $F(A)$ are closed under intersection and complement.  Moreover, given any finite automaton $\mathcal A$ over $\widetilde A$, there is a low-degree polynomial time algorithm to construct an automaton $\mathcal A'$ over $\widetilde A$ accepting precisely the reduced words representing elements of $\pi(L(\mathcal A))$.  It follows, that if $H_1,\ldots, H_n$ are finitely generated subgroups of $F(A)$, given by finite generating sets, and $g\in F(A)$, then one can effectively construct (in polynomial time in the sum of the lengths of the generators of the $H_i$ and the length of $g$) a finite automaton over $\widetilde A$ accepting precisely the reduced words representing an element of $gH_1\cdots H_n$, and similarly for finite unions of such subsets. See~\cite{Ste2} for details.

We will show that the nil-closure of a rational subset $L$ of $F(A)$ is a rational subset.  Moreover, there is an algorithm which, given a finite automaton $\mathcal A$ over $\widetilde A$ with $\pi(L(\mathcal A))=L$, produces a finite automaton $\mathcal A'$ over $\widetilde A$ accepting precisely the reduced words representing elements of the nil-closure of $L$.
%(Actually, we shall only give complete details for rational subsets of $A^*$ because the paper~\cite{Ste2}, unfortunately, did not state the necessary results for arbitrary rational subsets of a free group.)
The first step to do this is to construct from a set $H_1,\ldots, H_n$ of finitely generated subgroups of $F(A)$ (given by finite generating sets) a finite automaton accepting the reduced words in the nil-closure of $H_1\cdots H_n$.

In the seminal paper~\cite{Sta}, Stallings associated to each finitely generated subgroup $H$ of $F(A)$ an inverse automaton $\mathcal A(H)$ which can be used to solve a number of algorithmic problems concerning $H$ including the membership problem.  Stallings, in fact, used a different language than that of inverse automata; the automata theoretic formulation is from~\cite{Mar01}.

 An \emph{inverse automaton} $\mathcal A$ over $A$ is an $\widetilde A$-automaton with the property that there is at most one edge labeled by any letter leaving any vertex and if there is an edge $p\to q$ labeled by $a$, then there is an edge labeled by $a^{-1}$ from $q\to p$.  Moreover, we require that there is unique initial vertex, which is also the unique terminal vertex.  The set of all reduced words accepted by a finite inverse automaton is a finitely generated subgroup of $F(A)$ sometimes called the \emph{fundamental group} of the automaton.

If $H$ is a finitely generated subgroup of $F(A)$, then there is a unique finite connected inverse automaton $\mathcal A(H)$ whose fundamental group is $H$ with the property that all vertices have out-degree at least $2$ except possibly the initial vertex (where we recall that there are both $A$ and $A^{-1}$-edges). One description of $\mathcal  A(H)$ is as follows.  Take the inverse automaton $\mathcal A'(H)$ with vertex set the coset space $F(A)/H$ and with edges of the form $Hg\xrightarrow{\,\,a\,\,} Hga$ for $a\in \widetilde A$; the initial and terminal vertices are both $H$.  Then $\mathcal A(H)$ is the subautomaton whose vertices are cosets $Hu$ with $u$ a reduced word that is a prefix of the reduced form of some element $w$ of $H$ and with all edges between such vertices; the coset $H$ is still both initial and final.  Stallings presented an efficient algorithm to compute $\mathcal A(H)$ from any finite generating set of $H$ via a procedure known as folding. Conversely, one can efficiently compute a finite free basis for $H$ from $\mathcal A(H)$. From the construction, it is apparent that there is an automaton morphism $\mathcal A(H)\to \mathcal A(K)$ if and only if $H\subseteq K$ for finitely generated subgroups $H,K$.  Also, it is known that $H$ has finite index if and only if $\mathcal A(H)=\mathcal A'(H)$.  Stallings also provided an algorithm to compute $\mathcal A(H\cap K)$ from $\mathcal A(H)$ and $\mathcal A(K)$ (note that intersections of finitely generated subgroups of free groups are finitely generated by Howson's theorem).  See~\cite{Sta,Mar01,Ste2} for details.

%Let $Y=\{h_1,\ldots,h_m\}$ be a finite set of elements of $F(A)$ that generates $H$ and $h_i=h_{i,1}\cdots h_{i,m_i}$ with $h_{i,j}\in A\cup A^{-1}$ for $1\leq i\leq m$ and $1\leq j\leq m_i$.

%\begin{enumerate}
%\item[•] First, construct a set of $m$ circuits $c_i$ with edges $h_{i,1},\ldots, h_{i,m_i}$ around a distinguished vertex $1$, with the following convention: an inverse letter $a^{-1} (a \in A)$  gives rise to an $a$-labeled edge in the reverse direction on  the corresponding circuit.

%\item[•] Then, iteratively identify identically-labeled pairs of edges starting or ending at the same vertex.

%\item[•] Finally, iteratively remove vertices of degree one except 1.
%\end{enumerate}

%The graph $\mathcal{A}(H)$ is a reduced inverse automaton over $A$ in the following sense:
%the vertices of $\mathcal{A}(H)$ are considered as states, the distinguished vertex 1 as the initial-final state, and if $a \in A$ labels an edge from vertex $p$ to $q$, then we let $p. a = q$ and $q. a^{-1} = p$. If $u$ is a reduced word in $A\cup A^{-1}$, then $1.u = 1$ if and only if $u \in H$.

We next recall the notion of an overgroup of a finitely generated subgroup
of a free group from~\cite{Mar01}. Let $H$ be a finitely generated
subgroup of the free group $F(A)$ with Stallings automaton $\mathcal{A}(H)$.
Then, as a finite inverse automaton has only finitely many quotient automata, there are only finitely many subgroups $K$ of $F(A)$
containing $H$ such that the natural morphism from $\mathcal{A}(H)$ to
$\mathcal{A}(K)$ is onto. Such subgroups are called \emph{overgroups}
of $H$. Each overgroup is finitely generated and one can effectively
compute the set of Stallings automata of the  overgroups of $H$ from $\mathcal A(H)$. An important result, proved implicitly in~\cite{RiZa} and explicitly in~\cite{Mar01}, is that the $p$-closure of a
finitely generated subgroup $H$ is an overgroup of $H$ for any prime $p$~\cite[Corollary 2.4]{Mar01}.
%Now, suppose that $K$ is a finitely
%generated subgroup of the free group such that $H \subseteq K$. The
%set $Q(H,K)$ of all prime numbers $p$ such that $H$ is $p$-dense in
%$K$ is empty or cofinite; and it is effectively computable~\cite[Proposition 4.2]{Mar01}.

Let $K$ be a finitely generated subgroup of $F(A)$ and
let $\mathbb{P}(K)$ denote the set of prime
numbers $p$ such that $K$ is $p$-closed. Then the set $\mathbb{P}(K)$ is
either finite or cofinite, and it is effectively computable from $\mathcal A(H)$ by~\cite[Proposition 4.3]{Mar01}.  That is, one can decide if $\mathbb{P}(K)$ is finite or co-finite and  one can effectively list $\mathbb{P}(K)$ if it is finite and, otherwise, effectively list the complement of $\mathbb{P}(K)$, which in this case is finite.
%In fact, the proof of that result shows
%that $\mathbb{P}(H)$ is an effectively computable finite Boolean
%combination of finite and cofinite sets, namely the sets
%$\mathbb{P}(K)$ and $Q(H,K)$ where $K\supsetneq H$ is a strict overgroup of $H$ and so one can proceed by induction on the number of states of $\mathcal{A}(H)$.

Margolis, Sapir and Weil presented a procedure to compute the Stallings automaton of the nil-closure of a finitely generated subgroup of a free group (which is again finitely generated) as follows.

\begin{algorithm}[Margolis, Sapir, Weil]\label{a:msw}
Let $H$ be a finitely generated subgroup of $F(A)$ given by a finite generating set.
\begin{enumerate}
\item Compute $\mathcal A(H)$ using the Stallings folding algorithm.
\item Compute the set $\mathbb{A}$ of overgroups of $H$ (or more precisely, their Stallings automata).
\item Compute the subset $\mathbb{A}'=\{S\in \mathbb{A}\mid \mathbb{P}(S)\neq\emptyset\}$.
\item Compute the Stallings automaton of the intersection $K$ of the elements of $\mathbb{A}'$.
\item Return $\mathcal A(K)$ as the Stallings automaton of the nil-closure $K$ of $H$.
\end{enumerate}
\end{algorithm}

Now, let $H_1, \ldots, H_n$ be finitely generated subgroups of the free group $F(A)$.
By Proposition~\ref{04}, $\mathrm{Cl}_{nil}(H_1 \cdots H_n)$ is the
intersection of all the $p$-closures of $H_1 \cdots H_n$. Ribes and
Zalesski\u\i~ in~\cite{RiZa} proved that if the subgroups $H_1,
\ldots, H_n$ are $p$-closed for some prime $p$, then $H_1 \cdots H_n$
is $p$-closed too. Hence, the $p$-closure of $H_1 \cdots H_n$ is equal
to $\mathrm{Cl}_p(H_1)\mathrm{Cl}_p(H_2)\cdots \mathrm{Cl}_p(H_n)$.
By~\cite[Corollary 2.4]{Mar01} the $p$-closure of a finitely generated subgroup is one of its overgroups, which we shall exploit henceforth.

The procedure to compute an automaton accepting the reduced words in the nil-closure of $H_1 \cdots H_n$ is quite
similar to the procedure to compute the Stallings automaton of the nil-closure of a finitely generated subgroup in~\cite{Mar01}.

\begin{algorithm}\label{a:ouralg}
Let $H_1,\ldots, H_n$ be finitely generated subgroups of the free group $F(A)$ given by finite generating sets.
\begin{enumerate}
\item Compute the Stallings automata $\mathcal A(H_i)$ for $i=1,\ldots, n$.
\item Compute the set $\mathbb{A}_i$ of overgroups of $H_i$, for $1\leq i\leq n$, (or more precisely their Stallings automata).
\item Compute, for each set in the collection, \[\mathbb C=\{S_1 \cdots S_n \mid S_i\in \mathbb A_i\ \text{and}\ \mathbb P(S_1)\cap\cdots\cap \mathbb P(S_n)\neq \emptyset\}\] an $\widetilde A$-automaton accepting the set of reduced words belonging to it.
\item Compute an $\widetilde A$-automaton $\mathcal B$ accepting the reduced words in the intersection of the sets in the collection $\mathbb C$.
\item Return $\mathcal B$ as an $\widetilde A$-automaton accepting the reduced words in the nil-closure of $H_1\cdots H_n$.
\end{enumerate}
\end{algorithm}

We shall prove in a moment that Algorithm~\ref{a:ouralg} is correct.  First we verify that each of the steps of the algorithm can effectively be carried out.  The only step that is not straightforward to carry out based on the known algorithmic properties of rational subsets of free groups is deciding in (3) whether or not the set $\mathbb
P(S_1)\cap\cdots\cap \mathbb P(S_n)$ is empty  for $S_1\in \mathbb A_1, \ldots, S_n\in \mathbb A_n$.  Indeed, if $J\subseteq \{1,\ldots, n\}$ is the set of indices
such that $\mathbb P(S_i)$ is finite, then we can effectively compute $A=\bigcap_{i\in J}\mathbb P(S_i)$ and $B=\bigcup_{i\notin J}\mathbb P(S_i)'$ where $\mathbb P(S_i)'$ denotes the finite complement of $\mathbb P(S_i)$ for $i\notin J$.  Then $\mathbb P(S_1)\cap\cdots\cap \mathbb P(S_n)=\emptyset$ if and only if $A\subseteq B$, which is decidable as $A$ and $B$ are effectively computable finite sets.

%Let $A^{\ast}$ be the free monoid on $A$.
%Recall that a language $L \subseteq A^{\ast}$ is said to be \emph{rational} if it may be expressed in terms of the empty language and the languages of the %form $\{a\}$ with $a \in A$ by applying a finite number of times the binary operations of taking the union $L \cup K$ of two languages $L$ and $K$ or their concatenation $LK = \{uv \mid u \in L, v \in K\}$, or the unary operation of taking the submonoid $L^{\ast}$ generated by $L$.
%A language $L \subseteq A^{\ast}$ is recognized by a homomorphism $\varphi\colon A^{\ast}\rightarrow S$ into a semigroup $S$ if there exists a subset $P \subseteq S$ such that $L = \varphi^{-1}(P)$. Kleene~\cite{Kl} proved that a language $L$ over a finite alphabet is rational if and only if it is recognized by some finite automaton  and Myhill~\cite{My} proved that a language $L$ is recognized by a finite automaton if and only if it is recognized by a finite semigroup.
%
%Benois proved that membership in rational subsets of finitely generated free groups, such as products of finitely generated subgroups, is decidable~\cite{Beno} or~\cite{Ste2} for details. Hence the subsets $S_1\cdots S_n$ appearing in the third step of the algorithm are computable.  Actually, the intersection of rational subsets of a free group is again an effectively computable rational subset~\cite{Beno}, so one could compute an automaton testing membership in the intersection of $\mathbb C$.

Proposition~\ref{04} will be used to check that Algorithm~\ref{a:ouralg} is correct.

\begin{thm} \label{nil-dec}
The nil-closure of the product of finitely many finitely generated subgroups of a free group is an effectively computable rational subset and hence has a decidable membership problem.
\end{thm}
\begin{proof}
We verify correctness of Algorithm~\ref{a:ouralg}. Each element of $\mathbb C$ contains $H_1\cdots H_n$ and is $p$-closed for some prime $p$ and thus contains the nil-closure of $H_1\cdots H_n$ by Proposition~\ref{04}.  Therefore, the nil-closure of $H_1\cdots H_n$ is contained in the intersection of $\mathbb C$. On the other hand, for each prime $p$, one has that $\mathrm{Cl}_p(H_1\cdots H_n)=\mathrm{Cl}_p(H_1)\cdots \mathrm{Cl}_p(H_n)$ (by~\cite{RiZa}) appears in $\mathbb C$ by~\cite[Corollary 2.4]{Mar01}.  Thus the intersection of $\mathbb C$ is exactly the nil-closure of $H_1\cdots H_n$ by Proposition~\ref{04}.
\end{proof}

%Let $L\subseteq A^*$ be a rational subset and $\mathsf{H}$ an extension-closed pseudovariety of groups.  Assume that $L$ is given by an $A$-automaton.
%The $\mathsf{H}$-closure of the rational set $L$ can be written as a finite union of sets of the form $gG_{1}\cdots G_{n}$
%with $g \in F(A)$ and the $G_i$ finitely generated $\mathsf H$-closed subgroups, for every integer $1\leq i\leq n$, by~\cite[Proposition 6.24]{Ste2}.  Moreover, generating sets for the subgroups $G_i$ can be effectively computed as long as there is an algorithm that computes a finite generating set for the $\mathsf H$-closure of a finitely generated subgroup $H$ of a free group from a finite generating set of $H$.

Let $L\subseteq F(A)$ be a rational subset  given by an $\widetilde{A}$-automaton.  Then, by the results of~\cite{Pi-Ru} and~\cite{RiZa0},  the pro-$\mathsf G$ closure $L$ can be effectively written as a finite union of sets of the form $gG_{1}\cdots G_{n}$
with $g \in F(A)$ and the $G_i$ finitely generated subgroups, for every integer $1\leq i\leq n$.

\begin{cor} \label{nil-dec-rath}
The nil-closure of a rational subset of $F(A)$ is computable and rational. More precisely, there is an algorithm that, given an $\widetilde A$-au\-tom\-a\-ton accepting $L$, produces an $\widetilde A$-automaton accepting exactly the reduced words in the nil-closure of $L$.
\end{cor}
\begin{proof}
Let $L$ be a rational subset of $F(A)$. Then we have that $\mathrm{Cl}_{nil}(L)=\mathrm{Cl}_{nil}(\mathrm{Cl}_{\mathsf{G}}(L))$.
Now by~\cite{Pi-Ru} and~\cite{RiZa0},  there exist effectively computable finitely generated subgroups $G_{1,j},$ $\ldots,G_{r_j,j}$, for $1\leq j\leq s$ of $F(A)$ and elements $g_1,\ldots,g_s\in F(A)$ such that
\[\mathrm{Cl}_{\mathsf{G}}(L)=g_1G_{1,1}\cdots G_{r_1,1}\cup\ldots\cup g_sG_{1,s}\cdots G_{r_s,s}.\]
Hence, we have
\begin{align*}
\mathrm{Cl}_{nil}(L)&=\mathrm{Cl}_{nil}(g_1G_{1,1}\cdots G_{r_1,1})\cup\ldots\cup \mathrm{Cl}_{nil}(g_sG_{1,s}\cdots G_{r_s,s})\\
&=g_1\mathrm{Cl}_{nil}(G_{1,1}\cdots G_{r_1,1})\cup\ldots\cup g_s\mathrm{Cl}_{nil}(G_{1,s}\cdots G_{r_s,s}).
\end{align*}
It now follows from Theorem~\ref{nil-dec} and the known algorithmic properties of rational subsets of free groups, that we can construct an $\widetilde A$-automaton accepting precisely the reduced words in $\mathrm{Cl}_{nil}(L)$. This completes the proof.
\end{proof}

The next example shows that, for finitely generated subgroups $H_1, \ldots, H_n$ of $F$ with $n>1$, in general, the subset $H_1 \cup H_2 \cup \ldots \cup H_n$ may be $p$-dense, for every prime $p$, without being nil-dense and hence the nil-closure of a rational subset of a free group need not be the intersection of its $p$-closures over all primes $p$.

\begin{example}\label{p-nil-dense}
Consider the subgroups $H$ and $K$ of the free group $F$ from Example~\ref{ex:products.not.closed} and recall that the subgroup $H$ is $p$-dense for every prime except prime $2$ and $K$ is $p$-dense for every prime except prime $3$.  Thus $H\cup K$ is $p$-dense for every prime $p$.  On the other hand, we saw that $H$ and $K$ are both nil-closed in Example~\ref{ex:products.not.closed} and hence $H\cup K$ is nil-closed.  As $H\cup K$ is a proper subset of $F$, we conclude that $H\cup K$ is not nil-dense and hence is not the intersection of its $p$-closures over all primes $p$.  More explicitly, one can check that under the canonical projection $F\to \mathbb Z/6\mathbb Z\times \mathbb Z/6\mathbb Z$, the image of $H\cup K$ is proper.
\end{example}

%By Example~\ref{p-nil-dense}, we deduce that the nil-closure of a
%rational subset in general is not equal to the intersection over all
%primes $p$ of its $p$-closures.
We end this section by investigating different conditions under which the union of finitely many finitely generated subgroups of a free group is $p$-dense for a prime $p$.

\begin{prop} \label{03}
Let $F$ be a group and $H_1, \ldots, H_n$ be subgroups of $F$. If the subset $H_1 \cup H_2 \ldots \cup H_n$ is $p$-dense for some prime $p \geq n$ in $F$, then there exists a positive integer $i$ such that $H_i$ is $p$-dense.
\end{prop}

\begin{proof}
Since the subset $H_1 \cup H_2 \ldots \cup H_n$ is $p$-dense in $F$,
for every homomorphism $\varphi \colon F \rightarrow P$ onto a finite $p$-group $P$, $\varphi(H_1 \cup H_2 \ldots \cup H_n) =P$.
The group $\varphi(H_i)=P_i$ is a subgroup of $P$ and, thus, it is a $p$-group, for $1 \leq i \leq n$. Hence, $P=P_1\cup \ldots \cup P_n$.

Suppose that $P_i \neq P$ for every $1 \leq i \leq n$. Hence, $P$ is noncyclic and it is covered by its proper subgroups. Lemma 116.3.(a) of \cite{Ber3} yields $n \geq p + 1$.
%\bb{Lemma 116.3. Suppose that a noncyclic $p$-group $G$ of order $p^m$ is covered by $n$ proper subgroups $A_1,\ldots, A_n$ as $G = A_1\cup\ldots\cup A_n$. Then $n \geq p + 1$.}\\
This contradicts the assumption that $p \geq n$. Therefore, there exists a positive integer $i$ such that $\varphi(H_i)=P$.

Now suppose that the subgroup $H_i$ is not $p$-dense for every $1 \leq i \leq n$. Hence, there exist a $p$-group $P_i$ and an onto homomorphism $\varphi_i \colon F \rightarrow P_i$ such that $\varphi_i(H_i) \neq P_i$ for every $1 \leq i \leq n$.

Let $\varphi_1 \times \cdots \times \varphi_n\colon F\rightarrow P_1 \times \cdots \times P_n$.
Since $P_1 \times \cdots \times P_n$ is a $p$-group, $G=(\varphi_1 \times \cdots \times \varphi_n)(F)$ is a $p$-group and, thus,
by the above, there exists a positive integer $i$ such that
\[(\varphi_1 \times \cdots \times \varphi_n)(H_i)= G.\]  As $G$ is a subdirect product of $P_1,\ldots, P_n$,
it follows that $\varphi_i(H_i)= P_i$, a contradiction.
\end{proof}

 The next example shows that the hypothesis that $p \geq n$ in Proposition \ref{03} cannot be dropped. For this purpose, we recall that a group $G$ is said to be \emph{minimal non-Abelian} if it is non-Abelian but all its proper subgroups are Abelian \cite{Ber1}.

\begin{example}
Suppose that $P$ is a minimal non-Abelian $p$-group. By \cite[Lemma 116.1.(a)]{Ber3}, $P$ contains a set $X$ of $p+1$ pairwise noncommuting elements of $P$ and no subset with more than $p+1$ elements consists of pairwise noncommuting elements.
%\bb{Lemma 116.1. Let $G$ be a non-Abelian $p$-group. Then, if $G$ is minimal non-Abelian, then $\gamma(G)=p + 1$ ($\gamma(G)=\max\{\abs{M}\mid M \in \Delta(G)\}$, let $M$ be a maximal subset (with respect to inclusion) of pairwise noncommuting elements of a non-Abelian group $G$. We denote the set of all such subsets by $\Delta(G)$).}\\
%Choose such a subset $X$. For each $x\in X$ choose a maximal subgroup $M_x$ containing the subgroup $\langle x,Z(P)\rangle$.
%Note that the subgroup $\langle x,Z(P)\rangle$ is abelian and since $P$ is not abelian, there exists a maximal subgroup of $P$ such that containing it.
%Note that the mapping $x\mapsto M_x$ must be injective. Hence, $P$ has $p+1$ pairwise distinct maximal subgroups $M_1, \ldots, M_{p+1}$.
By \cite[Lemma 1.1 and Exercise 1, page 22]{Ber1}, we have that $[P:Z(P)]=p^2$.
%\bb{Suppose that $G$ is a finite group. If $p^k\mid \abs{G}$ then $G$ has a subgroup with order $p^k$ (1-5-5).\\
%Lemma 1.1. Let $A$ be an abelian subgroup of index $p$ of a non-Abelian $p$-group $G$. Then $\abs{G}= p\abs{G'}\abs{Z(G)}$.\\
%Exercise 1. If $G$ is a minimal non-Abelian $p$-group, then $\abs{G'}= p$ and $G=\langle x, y\rangle$ for some $x, y \in G$.}\\
It follows that, for each $x\in X$, the subgroup $M_x=\langle x,Z(P)\rangle$ is proper (since $P$ is non-Abelian) and of index $p$ (since $x\notin Z(P)$), hence maximal.  Moreover, if $x\neq y$, then $M_x\neq M_y$ as $M_x$ is Abelian and $y$ does not commute with $x$. It follows that $M_x\cap M_y=Z(P)$ for all $x\neq y$ in $X$.
Suppose that $\abs{P}=p^t$. As \[\left|Z(P)\cup \bigcup_{x\in X}M_x\setminus Z(P)\right| = p^{t-2}+(p+1)(p^{t-1}-p^{t-2})=p^t=|P|,\] the group $P$ is covered by its maximal subgroups $M_x$ with $x\in X$. Let $\varphi \colon F \rightarrow P$ be a homomorphism onto $P$, with $F$ a finitely generated free group, and put $P_x = \varphi^{-1}(M_x)$ for $x\in X$. The subgroup $P_x$ is $p$-open. Since $F= \bigcup_{x\in X}P_x$,  the set $\bigcup_{x\in X}P_i$ is $q$-dense for every prime $q$. But the subgroup $P_x$ is not $p$-dense for every $x\in X$, as $\varphi(P_x)=M_x\subsetneq P$.  Thus the bound in Proposition~\ref{03} is tight.
\end{example}

\section{Decidability of the pseudovariety $\mathsf{J} \ast \mathsf{G_{nil}}$}
The reader is referred to~\cite{Almeida:book,Rho-Ste} for basic definitions from finite semigroup theory.
Let $\mathsf{V}$ be a pseudovariety of monoids and $\mathsf{H}$ be a pseudovariety of groups.
If $\varphi\colon M \rightarrow H$ is a surjective morphism with $M$ a monoid and $H$ a group, then $N = \varphi^{-1}(1)$ is a submonoid of $M$. In this case, we say that $M$ is a co-extension of $H$ by $N$. Recall that the pseudovarieties $\mathsf{V}\ast\mathsf{H}$ and $\mathsf{V}\malcev \mathsf{H}$ are generated, respectively, by semidirect products of monoids in $\mathsf{V}$ with groups in $\mathsf{H}$ and by co-extensions of groups in $\mathsf{H}$ by monoids in $\mathsf{V}$. In general, $\mathsf{V}\ast\mathsf{H}\subseteq\mathsf{V}\malcev \mathsf{H}$ (by consideration of the semidirect product projection).

A \emph{relational morphism} $\varphi\colon M\to N$ of monoids is a relation such that $\varphi(m)\neq \emptyset$ for all $m\in M$ and $\varphi(m)\varphi(m')\subseteq \varphi(mm')$ for all $m,m'\in M$.
Recall that a subset $X \subseteq M$ of a finite monoid is called \emph{$\mathsf{H}$-pointlike} if, for every relational morphism $\varphi\colon M \rightarrow H$ with $H \in \mathsf{H}$, there exists $h \in H$ such that $X \subseteq \varphi^{-1}(h)$~\cite{Ste3}. For example the submonoid
of elements of $M$ that relate to $1$ under any relational morphism to a member of $\mathsf{H}$ is $\mathsf{H}$-pointlike; this submonoid is denoted $K_{\mathsf{H}}(M)$ and called the \emph{$\mathsf H$-kernel} of $M$.
An element $(m_1,\ldots,m_k) \in M^k$ is called an \emph{$\mathsf{H}$-liftable $k$-tuple} if, for every relational morphism $\mu \colon M \rightarrow H$ with $H\in \mathsf{H}$, there exist $h_1,\ldots, h_k \in H$ such that $h_1 \ldots h_k = 1$ and $h_i \in \mu(m_i)$ for all $1\leq i\leq k$~\cite{Ste2}. For example, $(m)$ is an $\mathsf H$-liftable $1$-tuple if and only if $m\in K_{\mathsf H}(M)$.

A finite monoid $M$ is a called a \emph{block group} if each element $a\in M$ has at most one generalized inverse, that is, there is at most one element $a'\in M$ such that $aa'a=a$ and $a'aa'=a'$.  The pseudovariety of block groups is denoted $\mathsf{BG}$~\cite{Rho-Ste}.  The power set of a finite group is a typical example of a block group and the pseudovariety of block groups is generated by power sets of finite groups~\cite{Hen-Mar-Pin-Rho}.

Let $\mathrm{reg}(M)$ be the set of regular elements of $M$; that is, those elements $a\in M$ for which $aa'a=a$ and $a'aa'=a'$ for some $a'\in M$. The pseudovariety $\Reg \mathsf{V}$ consists of all monoids $M$ such that $\mathrm{reg}(M)$ generates a monoid in $\mathsf{V}$.

In previous work, the third author~\cite{Ste3} showed that membership was decidable in certain pseudovarieties of the form $\mathsf V\malcev \mathsf{G_{nil}}$ without having computed membership in the $\mathsf{G_{nil}}$-kernel.
% Recall that a subgroup $H$ of the free group $F(A)$ is called \emph{$\mathsf{H}$-extendible} if $\mathcal{A}(H)$ can be embedded into $\mathcal A(K)$ for some $\mathsf{H}$-open subgroup $K$~\cite{Mar01}. The third author proved that if $\mathsf{V}$ is decidable,
%and one can decide if a finitely generated subgroup of a finite rank free group is $\mathsf{H}$-extendible, then $(\Reg\mathsf{V}) \malcev \mathsf{H}$ is decidable~\cite[Corollary 7.2]{Ste3}.
%Hence,
Namely, he proved that the pseudovariety $(\Reg \mathsf{V}) \malcev \mathsf{G_{nil}}$ is decidable for every decidable pseudovariety $\mathsf{V}$.% since $\mathsf{G_{nil}}$-extendibility was proved decidable in~\cite{Mar01}.

Let us recall that $\mathsf{A}$ denotes the pseudovariety of all finite aperiodic monoids, $\mathsf{DS}$ denotes the pseudovariety of all finite monoids whose regular $\mathcal{J}$-classes are subsemigroups, $\mathsf{DA}=\mathsf{DS}\cap \mathsf{A}$, and $\mathsf{J}$ denotes the pseudovariety of all finite $\mathcal{J}$-trivial monoids. For each of the pseudovarieties $\mathsf{A}, \mathsf{DS}$, $\mathsf{DA}$ and $\mathsf{J}$, we have $\Reg\mathsf{V}=\mathsf{V}$. Thus the pseudovarieties $\mathsf{A} \malcev \mathsf{G_{nil}}$, $\mathsf{DS} \malcev \mathsf{G_{nil}}$, $\mathsf{DA} \malcev \mathsf{G_{nil}}$ and $\mathsf{J} \malcev \mathsf{G_{nil}}$ are decidable by the results of~\cite{Ste2} and~\cite{Mar01}.

We shall now prove that the pseudovariety $\mathsf{V} \malcev \mathsf{G_{nil}}$ has decidable membership for every decidable pseudovariety $\mathsf{V}$. (Recall that a pseudovariety is called \emph{decidable} if it has a decidable membership problem.)
Before proving our claim, we recall a proposition from~\cite{Ste2} that establishes a relationship between
$\mathsf{H}$-liftable $k$-tuples and $\mathsf{H}$-closure of rational subsets. The reader should recall that if $M$ is a finite $A$-generated monoid (with $A$ finite), then, for any $m\in M$, then the language of words $w\in A^*$ mapping to $m$ (i.e., $[w]_M=m$) is rational in $A^*$ ~\cite{EilenbergA}, and hence in the free group $F(A)$.

\begin{prop}\cite[Proposition 7.20]{Ste2}\label{Block_2}
Let $M$ be a finite $A$-generated monoid, $(m_1,\ldots,m_k) \in M^k$, $L_{m_i} = \{u \in A^{\ast}\mid [u]_M = m_i\}$, and $\mathsf{H}$ be a pseudovariety of groups. Then $(m_1,\ldots,m_k)$ is an $\mathsf{H}$-liftable $k$-tuple if and only if $1 \in \mathrm{Cl}_{\mathsf{H}}(L_{m_1}\cdots L_{m_k})$ where the closure is taken in the free group $F(A)$ on $A$.
\end{prop}

\begin{thm}\label{Ker-nil}
The nilpotent kernel $K_{\mathsf{G_{nil}}}(M)$ of a finite monoid $M$ is computable.
\end{thm}

\begin{proof}
Fix a finite generating set $A$ for $M$.
We have that $m\in K_{\mathsf{G_{nil}}}(M)$ if and only if $(m)$ is a $\mathsf{G_{nil}}$-liftable 1-tuple.
By Proposition~\ref{Block_2}, $(m)$ is a $\mathsf{G_{nil}}$-liftable 1-tuple if and only if $1 \in \mathrm{Cl}_{nil}(L_{m})$ in $F(A)$.
By Corollary~\ref{nil-dec-rath}, $\mathrm{Cl}_{nil}(L_{m})$ is an effectively computable rational subset of $F(A)$ and the result follows.
\end{proof}

This yields the following corollary.

\begin{cor}
Let $\mathsf{V}$ be a decidable pseudovariety. Then the pseudovariety $\mathsf{V} \malcev \mathsf{G_{nil}}$ is decidable.
\end{cor}

\begin{proof}
It is well known that  $M \in \mathsf{V}\malcev \mathsf{G_{nil}}$ if and only if $K_{\mathsf{G_{nil}}}(M)\in \mathsf{V}$ (cf.~\cite[Theorem 3.4]{Hen-Mar-Pin-Rho}).
As $\mathsf{V}$ has decidable membership and $K_{\mathsf{G_{nil}}}(M)$ is computable by Theorem~\ref{Ker-nil}, the pseudovariety $\mathsf{V} \malcev \mathsf{G_{nil}}$ is decidable.
\end{proof}

Recall that $\mathsf{CR}$ denotes the pseudovariety of all finite completely regular monoids (that is, monoids satisfying an identity of the form $x^m=x$ with $m>1$).
In particular, the pseudovariety $\mathsf{CR} \malcev \mathsf{G_{nil}}$ is decidable.
It is well known, see for instance~\cite{Hen-Mar-Pin-Rho}, that, if $\mathsf{V}$ is local in the sense of Tilson~\cite{Tilson}, then $\mathsf{V} \ast \mathsf{H} = \mathsf{V} \malcev \mathsf{H}$. The pseudovarieties $\mathsf{A}$, $\mathsf{CR}$~\cite{Peter}, $\mathsf{DS}$~\cite{Jo-Tr} and $\mathsf{DA}$~\cite{Almeida:1996c} are local.
Therefore, the pseudovarieties $\mathsf{CR}\ast \mathsf{G_{nil}}$, $\mathsf{A}\ast \mathsf{G_{nil}}$, $\mathsf{DS}\ast \mathsf{G_{nil}}$ and $\mathsf{DA}\ast \mathsf{G_{nil}}$ are all decidable, where the last three of these results were already proved in~\cite[Corollary 7.3]{Ste3}.

Auinger and the third author defined arboreous pseudovarieties of groups in terms of certain properties of their relatively free profinite groups. This definition and more details can be found in~\cite{Auinger&Steinberg:2001a}.
They proved that a pseudovariety of groups $\mathsf{H}$ is arboreous if and only if $\mathsf{J}\malcev \mathsf{H} = \mathsf{J}\ast\mathsf{H}$~\cite[Theorem 8.3]{Auinger&Steinberg:2001a}. They also noted that each arboreous pseudovariety is join irreducible~\cite[Corollary 2.13]{Auinger&Steinberg:2001a} where a pseudovariety $\mathsf{H}$ is join irreducible if $\mathsf{H} =  \mathsf{H_1}\vee \mathsf{H_2}$ implies that $\mathsf{H} =\mathsf{H_1}$ or $\mathsf{H} = \mathsf{H_2}$.

Since the pseudovariety $\mathsf{G_{nil}}$ is not join irreducible, it is not arboreous. Thus $\mathsf{J}\malcev \mathsf{G_{nil}}\neq\mathsf{J}\ast \mathsf{G_{nil}}$. As we have already seen, the pseudovariety $\mathsf{J}\malcev \mathsf{G_{nil}}$ is decidable (cf.~\cite[Corollary 8.1]{Ste3}). We now show that the pseudovariety $\mathsf{J}\ast \mathsf{G_{nil}}$ is also decidable, a new result. Before proving this, we shall first recall two results from~\cite{Ste3}.

\begin{thm}\cite[Theorem 8.2]{Ste3}\label{Block_1}
Let $M$ be a block group and $\mathsf{H}$ a pseudovariety of groups. Then $M\in \mathsf{J}\ast\mathsf{H}$ if and only if, for every pair $\{\alpha,\beta\}$ of regular elements of $M$ which form an $\mathsf{H}$-pointlike set, one has that $\alpha\alpha^{-1}\beta\beta^{-1} = \alpha\beta^{-1}$.
\end{thm}

%Pin and Reutenauer~\cite{Pi-Ru}
%proved that the closure of a rational subset of the free group is a finite union of sets of the form $gG_1 G_2 \cdots G_r$, where $g\in F(A)$ and $G_1,\ldots,G_r$ are finitely generated subgroups of $F(A)$~\cite[Theorem 2.4]{Pi-Ru}. The third author connected the rational language associated to a regular element of a finite monoid to the corresponding language recognized by its Sch\"utzenberger graph. For a finite $A$-generated monoid $M$ and $m \in \mathrm{reg}(M)$, one can effectively compute a finitely generated subgroup $H\subseteq F(A)$ and $w\in F(A)$ such that $\mathrm{Cl}_{\mathsf{G}}(L_{m})=Hw$~\cite[Corollary 7.16]{Ste2}.

\begin{lem}\cite[Lemma 7.23]{Ste2}\label{Block_3}
Let $\mathsf{H}$ be a pseudovariety of groups, $M$ a block group, and $\alpha, \beta \in \mathrm{reg} (M)$. Then $\{\alpha, \beta\}$ is $\mathsf{H}$-pointlike if and only if $(\alpha, \beta^{-1})$ is a $\mathsf{H}$-liftable 2-tuple.
\end{lem}

Now by Proposition~\ref{Block_2}, Theorem~\ref{Block_1},  and Lemma~\ref{Block_3}, the following result can be seen as a corollary of the decidability of nil-closure in the free group of a regular language.

\begin{thm}\label{Main_2}
The pseudovariety $\mathsf{J}\ast \mathsf{G_{nil}}$ is decidable.
\end{thm}

\begin{proof}
It is well known~\cite{Hen-Mar-Pin-Rho} that $\mathsf{J}\malcev \mathsf{G}= \mathsf{BG}$ and $\mathsf{J}\ast\mathsf{G}\subseteq \mathsf{J}\malcev \mathsf{G}$. Thus $\mathsf{J}\ast\mathsf{G}_{nil}\subseteq \mathsf{BG}$ and so it suffices to decide membership in $\mathsf J\ast \mathsf{G_{nil}}$ for block groups.
Suppose that $M$ is a block group generated by a finite set $A$. By Theorem~\ref{Block_1}, $M\in \mathsf{J}\ast\mathsf{G_{nil}}$ if and only if, for every pair $\{\alpha,\beta\}$ of regular elements of $M$ which form a $\mathsf{G_{nil}}$-pointlike set, the equality $\alpha\alpha^{-1}\beta\beta^{-1} = \alpha\beta^{-1}$ holds. Thus it suffices  to decide whether a pair $\{\alpha,\beta\}$ of regular elements of $M$ is a $\mathsf{G_{nil}}$-pointlike set. In fact, by Proposition~\ref{Block_2} and Lemma~\ref{Block_3}, it suffices to decide whether $1\in \mathrm{Cl}_{nil}(L_{\alpha} L_{\beta^{-1}})$.  But since $L_{\alpha}L_{\beta^{-1}}$ is an effectively computable rational subset   of $A^*$, we can decide this by Corollary~\ref{nil-dec-rath}.  This completes the proof.
%
%
%
% We have $\mathrm{Cl}_{nil}(L_{\alpha} L_{\beta^{-1}})=\mathrm{Cl}_{nil}(\mathrm{Cl}_{\mathsf{G}}(L_{\alpha} L_{\beta^{-1}}))$.\\
%By the observation preceding the statement of the theorem, there exist effectively computable
%finitely generated subgroups $H$ and $K$ and words $w,u \in F(A)$ such that $\mathrm{Cl}_{\mathsf{G}}(L_{\alpha})=Hw$ and %$\mathrm{Cl}_{\mathsf{G}}(L_{\beta^{-1}})=Ku$. It follows that $$\mathrm{Cl}_{\mathsf{G}}(L_{\alpha}) \mathrm{Cl}_{\mathsf{G}}(L_{\beta^{-1}})=Hw Ku=wuu^{-1}w^{-1}Hwuu^{-1}Ku=wuH'K',$$
%whence, $\mathrm{Cl}_{nil}(L_{\alpha} L_{\beta^{-1}})=wu~\mathrm{Cl}_{nil}(H'K')$.
%Now, as membership in the nil-closure of the product of two subgroups is decidable, the result follows.
\end{proof}

An interesting open question is whether $\mathsf {PG_{nil}}$ has a decidable membership problem, where we recall that if $\mathsf H$ is a pseudovariety of groups, then $\mathsf{PH}$ is the pseudovariety generated by all power sets of groups in $\mathsf H$ (with setwise product).  It is known that $\mathsf{PH}\subseteq \mathsf J\ast \mathsf H$ and that equality holds precisely for the so-called Hall pseuodovarieties~\cite{PH}.  Hall pseudovarieties are join irreducible~\cite{PH} and so $\mathsf{PG_{nil}}\neq \mathsf J\ast \mathsf{G_{nil}}$.

\subsubsection*{Acknowledgments.}
The work of the first and second authors was supported, in part, by CMUP
(UID/MAT/00144/2013), which is funded by FCT (Portugal)
with national (MCTES) and European structural funds (FEDER), under the
partnership agreement PT2020.
The work of the second author was also partly supported by the FCT
post-doctoral scholarship SFRH/ BPD/89812/2012.
The third author was supported in part by a grant from the Simons Foundation (\#245268 to Benjamin Steinberg), the Binational Science Foundation of Israel and the US (\#2012080), a CUNY Collaborative Research Incentive Grant \#2123, by a PSC-CUNY grant and by NSA MSP grant \#H98230-16-1-0047.

%%%%%%%%%%%%%%%%%%%%%%%%%%%%%%%%%%%%%%%%%%%%%%%%%%%%%%%%%%%%%%%%%%%%%%%%%%%%%%%%%%%%%%%%%%%%%%%%%
%%%%%%%%%%%%%%%%%%%%%%%%%%%%%%%%%%%%%%%%%%%%%%%%%%%%%%%%%%%%%%%%%%%%%%%%%%%%%%%%%%%%%%%%%%%%%%%%%
%\bibliographystyle{plain}
%\bibliography{ref-pro-nil}
\def\cprime{$'$}

\end{document}